\documentclass[12pt]{amsart}
\usepackage{amsmath,amssymb,amsthm}
\usepackage{enumerate}
\usepackage{graphicx}
\usepackage{amsfonts}
\usepackage{bigints}
\usepackage{amsfonts}
\usepackage{mathptmx}      
\usepackage{lipsum}
\usepackage{graphicx}
\usepackage{amssymb}
\usepackage{amsmath}
\usepackage{bigints}
\usepackage{mathptmx}      

\newcommand{\overbar}[1]{\mkern 1.5mu\overline{\mkern-1.5mu#1\mkern-1.5mu}\mkern 1.5mu}

\theoremstyle{plain}
\newtheorem{theorem}{Theorem}[section]

\newtheorem{corollary}[theorem]{Corollary}
\newtheorem{lemma}[theorem]{Lemma}

\theoremstyle{definition}

\makeatletter
\@addtoreset{equation}{section}
\makeatother

\makeatletter
\newcommand{\Spvek}[2][r]{%
	\gdef\@VORNE{1}
	\left(\hskip-\arraycolsep%
	\begin{array}{#1}\vekSp@lten{#2}\end{array}%
	\hskip-\arraycolsep\right)}

\def\vekSp@lten#1{\xvekSp@lten#1;vekL@stLine;}
\def\vekL@stLine{vekL@stLine}
\def\xvekSp@lten#1;{\def\temp{#1}%
	\ifx\temp\vekL@stLine
	\else
	\ifnum\@VORNE=1\gdef\@VORNE{0}
	\else\@arraycr\fi%
	#1%
	\expandafter\xvekSp@lten
	\fi}
\makeatother
\begin{document}
	\title[Szeg\H{o}'S Condition on compact subsets of $\mathbb{C}$] {Szeg\H{o}'s condition on compact subsets of $\mathbb{C}$}
	\author{G\"{o}kalp Alpan}
\address{Department of Mathematics, Rice University, Houston, TX 77005, USA}
\email{alpan@rice.edu}
\begin{abstract}
Let $K$ be a non-polar compact subset  of $\mathbb{C}$ and $\mu_K$ be its equilibrium measure. Let $\mu$ be a unit Borel measure supported on a compact set which contains the support of $\mu_K$. We prove that a Szeg\H{o} condition in terms of the Radon-Nikodym derivative of $\mu$ with respect to $\mu_K$ implies that $$\inf_n \frac{\|P_n(\cdot;\mu)\|_{L^2(\mathbb{C};\mu)}}{\mathrm{Cap}(K)^n}>0.$$

We show that $\frac{\|P_n(\cdot;\mu_K)\|_{L^2(\mathbb{C};\mu_K)}}{\mathrm{Cap}(K)^n}\geq 1$ for any compact non-polar set $K$. We also prove that under an additional assumption, unboundedness of the sequence $\left(\frac{\|P_n(\cdot;\mu_K)\|_{L^2(\mathbb{C};\mu_K)}}{\mathrm{Cap}(K)^n}\right)$ implies that $K$ satisfies the  Parreau-Widom condition.

	
\end{abstract}
\maketitle
\section{Introduction}
Let $\mu$ be a unit Borel measure with an infinite compact support on $\mathbb{C}$. We denote by $P_n(z;\mu)$  the $n$-th degree monic orthonormal polynomial associated with $\mu$, i.e.,
\begin{equation}\label{infli}
 \| P_n(\cdot;\mu) \|_{L^2(\mathbb{C};\mu)}= \inf_{Q\in \mathcal{P}_n} \|Q\|_{L^2(\mathbb{C};\mu)}
\end{equation}

where  $\mathcal{P}_n$ is the set of all $n$-th degree monic (complex) polynomials and $\|\cdot \|_{L^2(\mathbb{C};\mu)}$ denotes the $L^2$ norm associated with $\mu$.

Let $d\mu(\theta)= w(\theta )\frac{d\theta}{2\pi}+ d\mu_s(\theta)$ be a unit Borel measure on $ \mathbb{T}$ whose support contains infinitely many points. Then Szeg\H{o}'s theorem (see Theorem 2.12.7 in \cite{Sim3} and p. 306 in \cite{szeg}) reads as follows:
\begin{equation}\label{loglu1}
\lim_{n\rightarrow\infty} \|P_n(\cdot;\mu)\|_{L^2(\mathbb{C};\mu)}^2= e^{\int \log{w(\theta)} \frac{d\theta}{2\pi}}.
\end{equation}

Note that, in \eqref{loglu1} the integral is either finite or $-\infty$.

For a general treatment of logarithmic potential theory, see e.g. \cite{Ransford}, \cite{saff}. Let us denote the logarithmic capacity by $\mathrm{Cap}(\cdot)$. For a non-polar compact subset $K$ of $\mathbb{C}$, we denote the equilibrium measure of $K$ by $\mu_K$. For the component of $\overline{\mathbb{C}}\setminus K$ that contains $\infty$, we use $\Omega_K$. By $g_{\Omega_K}(z):= g_{\Omega_K}(z;\infty)$, we mean the Green function for the domain $\Omega_K$ at infinity. If $K$ is a regular compact subset of $\mathbb{C}$ with respect to the Dirichlet problem, let $$\mathrm{PW}(K):=\sum g_{\Omega_K}(c_n)$$
where $\{c_n\}_n$ is the set of critical points of $g_{\Omega_K}$ counting multiplicity. If $\mathrm{PW}(K)<\infty$ then $K$ is called a Parreau-Widom set and $\Omega_K$ is called a Parreau-Widom domain.

Let $W_n(\mu):=\frac{\|P_n(\cdot; \mu)\|_{L^2(\mathbb{C};\mu)}}{\mathrm{Cap(supp(\mu))}^n}$ where $\mathrm{supp}(\cdot)$ denotes the support of $\mu$. Since $\mathrm{Cap(supp(\mu_K))}=\mathrm{Cap}(K)$ (see Lemma 1.2.7 in \cite{Stahl}) we use these two expressions interchangeably.

For the single interval case we have an analogue of \eqref{loglu1}:

\begin{theorem}\label{Widdo}[Theorem 12.3, Theorem 9.2 \cite{widom2}]
	Let $K$ be a compact non-degenerate interval on $\mathbb{R}$ and  $d\mu(x)= f(x) dx$ be a unit Borel measure supported on $K$ such that 
	
	$$R(\infty):= \exp{\left(\int_K \log{f}(x)\,d\mu_K(x)\right)}>0.$$
	
	Then $\lim_{n\rightarrow\infty} (W_n(\mu))^2= 2\pi R(\infty) \mathrm{Cap}(K)$.

\end{theorem}
If $K$ is a union of finitely many intervals then behavior of $W_n(\mu)$ is more irregular. See \cite{Chris} for a generalization and a more complete version of Theorem \ref{Widdo} when $K$ is a union of finitely many intervals.

Let $K$ be regular Parreau-Widom subset of $\mathbb{R}$. The Lebesgue measure $dx_{\restriction K}$ restricted to $K$ and $d\mu_K$ are mutually absolutely continuous, see \cite{sodin}. In particular this implies that $K$ has positive Lebesgue measure. This notion includes union of finitely many intervals as well as some Cantor sets, see e.g. \cite{alpgon2, peher}. The Szeg\H{o} theorem on Parreau-Widom sets is as follows \cite{christiansen}:

\begin{theorem}\label{chrisszego}
	Let $K$ be regular Parreau-Widom subset of $\mathbb{R}$. Let $d\mu= f(x) dx+ d\mu_s$ where $d\mu_s$ is the singular part with respect to the Lebesgue measure and suppose that the support of $\mu$, except possibly the isolated point masses, is equal to $K$. Denote the set of isolated points of the support by $\{x_k\}$. On condition that $$\sum_k {g_{\Omega_K}(x_k)}<\infty$$
	we have 
	\begin{equation}\label{chrsze}
	\limsup_{n\rightarrow\infty} W_n(\mu)>0 \iff \int_K \log{f}(x)\, d \mu_K(x)>-\infty.
	\end{equation}
	If one of the conditions in \eqref{chrsze} holds then 
	\begin{equation*}
	0< \liminf_{n\rightarrow\infty} W_n(\mu)\leq \limsup_{n\rightarrow\infty} W_n(\mu)<\infty.
	\end{equation*}

\end{theorem}

	The following definition is suggested in \cite{christi}: Let $K$ be a regular Parreau-Widom subset of $\mathbb{R}$. The Szeg\H{o} class of measures on $K$ is the set of all unit Borel measures 
$d\mu= f(x)dx+ d\mu_s$ such that
\begin{enumerate}[(i)]
	\item the essential support is equal to $K$.
	\item $\int_K \log{f(x)}\,d\mu_K(x)>-\infty.$ (Szeg\H{o} condition)
	\item the isolated points $\{x_n\}$ of $\mathrm{supp}(\mu)$
	satisfy $\sum_n g_{\Omega_K}(x_n)<\infty.$
\end{enumerate}
By Thoerem \ref{chrisszego}, (ii) can be replaced by one of the following conditions:
\begin{enumerate}[(ii$^\prime$)]
	\item $\limsup_{n\rightarrow\infty}W_n(\mu)>0.$ (Widom condition 1)
\end{enumerate}
\begin{enumerate}[(ii$^{\prime\prime}$)]
	\item $\liminf_{n\rightarrow\infty}W_n(\mu)>0.$ (Widom condition 2)
\end{enumerate}

In \cite{kruger}, Kr{\"u}ger and Simon discuss what Szeg\H{o} class of measures may mean if the prescribed set is of zero Lebesgue measure.  Generalizing the Szeg\H{o} class can help to generalize the notion of isospectral torus of Jacobi operators. We refer the reader to \cite{christi} for an exposition of the link between these two concepts on certain Parreau-Widom sets. 

A replacement of the condition (ii) is suggested in \cite{alpgon} for a regular compact subset $K$ of $\mathbb{R}$:
\begin{enumerate}[(ii$^*$)]
\item $\int_K \log{\frac{d\mu}{d\mu_K}(x)}\,d\mu_K(x)>-\infty.$
\end{enumerate}

This condition is equivalent to the original Szeg\H{o} condition on Parreau-Widom  subsets of $\mathbb{R}$.  Since $dx_{\restriction K}$  and $d\mu_K$ are mutually absolutely continuous,

$$\frac{d\mu}{dx} = \frac{d\mu}{d\mu_K} \frac{d\mu_K}{dx}$$

on $K$. Thus, in view of the inequality (see Section 4  in \cite{christiansen})
$$\int_K \log{\frac{d\mu_K}{d x}}(x)\,d\mu_K(x)>-\infty,$$
	
 (ii$^*$) and the conditions (i), (iii) can replace (i), (ii), (iii) without changing the definition.

  All of the previous generalizations (\cite{widom2}, \cite{peher}, \cite{Chris}, \cite{christiansen}) of the Szeg\H{o} theorem involve a Parreau-Widom set. The main reason of defining (ii$^*$) is that this condition can be used on arbitrary non-polar  compact subsets of $\mathbb{C}$ as opposed to $(ii)$. Note that, if the set has zero Lebesgue measure then the condition (ii) cannot be used to distinguish measures but the Widom conditions are still meaningful. Our main result shows that $(ii^*)$ is the natural generalization of the Szeg\H{o} condition $(ii)$ because if we assume $(ii^*)$ then the $W_n$'s have a lower bound in terms of a Szeg\H{o} integral:

\begin{theorem} \label{maintheo}
	Let $K$ be a non-polar compact subset of $\mathbb{C}$ and let $\mu$ be a unit Borel measure supported on a compact set containing $\partial\Omega_K$. Let $\mu_s$ denote the singular part of $\mu$ with respect to $\mu_K$ and $h$ be a non-negative measurable function on $\partial \Omega_K$ such that 
	\begin{itemize}
		\item $d\mu= h \,d{\mu}_K+ d\mu_s $.
		\item $M:=\int \log{h}\, d\mu_K >-\infty.$
		\item $\mathrm{Cap(supp(\mu))}= \mathrm{Cap(supp(\mu_K))}$.
		
	\end{itemize}
	
	Then $\inf_{n\in\mathbb{N}} (W_n(\mu))^2 \geq e^M.$
\end{theorem}

	Kr{\"u}ger and Simon \cite{kruger} study the orthogonal polynomials for the Cantor measure $\nu_{K_0}$. It is the Hausdorff measure for $h(t)=t^{\frac{\log{2}}{\log{3}}}$ restricted to the Cantor ternary set $K_0$. They pose the following conjecture in view of the numerical results:
	$$\liminf_{n\rightarrow\infty} W_n(\nu_{K_0})>0.$$
	If this is true then $\nu_{K_0}$ satisfies $(i),(ii^{''}),(iii)$ on $K_0$. But $\nu_{K_0}$ and $\mu_{K_0}$ are mutually singular by \cite{maka}. Thus $\nu_{K_0}$ does not satisfy the Szeg\H{o} condition $(ii^*)$. Cantor ternary set is regular, see \cite{biavol}. This means that if the conjecture above is true then the Widom condition (ii$^{\prime\prime}$)  does not imply the Szeg\H{o} condition $(ii^*)$ unlike the Parreau-Widom case. 

For a given compact infinite set $K$ in $\mathbb{C}$, the polynomial $T_{n,K}(z)=z^n+\cdots$ satisfying $$\|T_{n,K}\|_K=\min\{\|Q_n\|_{K}: \mbox{$Q_n$ monic polynomial of degree $n$}\}$$
is called the $n$-th Chebyshev polynomial for $K$ where $\|\cdot\|_{K}$ is the sup norm on $K$.

For a non-polar compact set $K\subset\mathbb{C}$, let $$M_{n,K}:= \|T_{n,K}\|_{K}/\mathrm{Cap}(K)^n.$$ For a review of the recent results for these ratios we refer the reader to \cite{Csz1} and many basic results regarding the asymptotics of $L^2$ and $L^\infty$ extremal polynomials can be found in \cite{simon1}.

The following result is a generalization of Theorem 3 in \cite{alpeq}:

\begin{corollary}\label{equi}
	Let $K$ be a non-polar compact subset of $\mathbb{C}$. Then $\inf_{n\in\mathbb{N}} W_n(\mu_K)\geq 1$. The inequality is sharp: If $K=  \mathbb{T}$ then $d\mu_K = d\theta / (2\pi)$ and $W_n (\mu_{ \mathbb{T}}) = 1 $ for all $n\in\mathbb{N}$.
\end{corollary}

\emph{Remark. } 
We would like to draw the reader's attention to the similarity between the general results regarding the lower bounds of $W_n(\mu_K)$ and $M_{n,K}$: It is well known that $M_{n,K}\geq 1$ , see Theorem 5.5.4 in \cite{Ransford} and the equality  is obtained for all $n$ on the unit circle.

It seems that the relation between the Szeg\H{o} condition and boundedness of $W_n(\mu)$ from above is more complicated. There are examples of Cantor sets $K(\gamma)$ such that \\ $W_n(\mu_{K(\gamma)})\rightarrow\infty$ as $n\rightarrow\infty$, see Example 5.3, \cite{alpgon}. We emphasize that $K(\gamma)$ does not satisfy the condition regarding the density of characters given in Theorem \ref{ratfreq}, see Section 4 in \cite{g1}. This condition is introduced recently and one can construct Parreau-Widom sets satisfying this density condition, see \cite{Chszy}. It is an open problem whether we can omit this property in Theorem \ref{ratfreq}. The next result implies Theorem 1.4 in \cite{Chszy} in view of \eqref{chebbb} and the proof is very similar.

\begin{theorem} \label{ratfreq}
	Let $K$ be a regular compact subset of $\mathbb{C}$. Suppose that $\{\chi_K^n\}_{n=-\infty}^\infty$ is dense in $\Pi(\Omega_K)^*$.
	If $(W_n(\mu_K))_{n=1}^\infty$ is bounded then $K$ is a Parreau-Widom set.
\end{theorem}

 As a corollary of Theorem \ref{ratfreq} we obtain the following result which complements Corollary \ref{equi} but the scope of Corollary \ref{ratfreq2} is much more limited. The proof of one of the implications is quite trivial and the inverse implication  follows from  Theorem 1.4 in \cite{Csz1} and Theorem \ref{ratfreq}. 

\begin{corollary}\label{ratfreq2}

		Let $K$ be a regular compact subset of $\mathbb{R}$ and $\{\chi_K^n\}_{n=-\infty}^\infty$ be dense in $\Pi(\Omega_K)^*$.
		Then $(W_n(\mu_K))_{n=1}^\infty$ is bounded if and only if $(M_{n,K})_{n=1}^\infty$ is bounded.
\end{corollary}

The plan of the paper is as follows. In section 2, we discuss preliminary results. In Sections 3, we prove the theorems.  In Appendix we prove a special case of Theorem \ref{maintheo} with a completely different method.

\section{preliminaries}

Let $K$ be a non-polar compact subset of $\mathbb{C}$. We denote the harmonic measure for $\Omega_K$ at $z$ by $w_{\Omega_K}(z;\cdot)$. The harmonic measure $w_{\Omega_K}(\infty;\cdot)$ at infinity is $\mu_K$, see Theorem 4.3.14, \cite{Ransford}. If $f$ is a Borel measurable function on $\partial\Omega _K$ such that
\begin{equation}\label{diri1}
H_{\Omega_K}(z;f):= \int_{\partial \Omega_K} f \,dw_{\Omega_K}(z;\cdot)
\end{equation}
is integrable for some $z\in \Omega_K$ then the integral in \eqref{diri1} is finite for all $z\in\Omega_K$, see Appendix A.3 in \cite{saff}. In this case $H_{\Omega_K}(z;f)$ is a harmonic function on $\Omega_K$ and it is called the solution of the Dirichlet problem corresponding to $f$ and $\Omega_K$. If additionally, $f$ is continuous and real valued on $\partial \Omega_K$ and $K$ is regular with respect to the Dirichlet problem then $\lim_{z\rightarrow\zeta} H_{\Omega_K}(z;f)= f(\zeta)$ for all $\zeta\in \partial \Omega_K$, see Corollary 4.1.8 in \cite{Ransford}. Hence if $f$ is continuous and $K$ is regular then $H_{\Omega_K}(\cdot;f)$ can be extended continuously to $\overbar{\Omega_K}$. We denote the extension by $H_{\overline{\Omega_K}}(z;f)$.

In \cite{widom2} asymptotics of orthogonal polynomials are given in terms of multiplicative analytic functions. Following \cite{Hasumi} (see p. 23-31), we say that $F$ is a multiplicative analytic (resp. meromorphic) function on $\Omega_K$ if $F$ is a multivalued analytic function on $\Omega_K$ with single valued absolute value $|F(z)|$. Each multiplicative analytic function determines a unique character: Let us fix a base point  $\mathcal{O}\in\Omega_K$. Let $F_\mathcal{O}$ be a single valued branch of $F$ at $\mathcal{O}$ and $c$ be a closed curve in $\Omega_K$ issuing from $\mathcal{O}$. Then $F_\mathcal{O}$ can be analytically continued along $c$ and the resulting function element at $\mathcal{O}$ is equal to $\zeta_F(c) F_\mathcal{O}$ where $|\zeta_F(c)|=1$. Note that the value of $\zeta_F$ is the same for homotopic curves and it is independent of the base point. Besides, if $c_1$ and $c_2$ are two closed curves issuing from $\mathcal{O}$ then $\zeta_F(c_1 c_2) F_\mathcal{O}= \zeta_F(c_1) \zeta_F(c_2) F_\mathcal{O}$. Thus, $\zeta_F(\cdot)$ is a character of the fundamental group $\Pi (\Omega_K)$. We denote the character group by $\Pi (\Omega_K)^\star$.

Multiplicative analytic functions can be defined in terms of analytic sections of line bundles. The following definitions about analytic sections are represented for the convenience of the reader. See Ch. 2 in \cite{Hasumi} and Section 1 of \cite{Widomhp} for more details.

Let $K$ be a regular compact subset of $\mathbb{C}$ and let $\mathcal{V}= \{V_i: i\in I\}$ be an open covering of $\Omega_K$. Then  $(\{\zeta_{ij}\}, \{V_i\})_{{i,j\in I}}$  is called a 1-cocycle over $\mathcal{V}$ if  $\zeta_{ij}\in \mathbb{T}$ and if 
$$ \zeta_{ij} \zeta_{jk}=\zeta_{ik}$$
whenever $V_i\cap V_j \cap V_k\neq\emptyset$. 

Let $\mathcal{V}_1= \{V_{i_1}\}$ and $\mathcal{V}_2= \{V_{i_2}\}$ be two open coverings of $\Omega_K$. Then two 1-cocycles \\
$(\{\zeta_{i_1j_1}\}, \{V_{i_1}\})$ and $(\{\zeta_{i_2 j_2}\}, \{V_{i_2}\})$ are equivalent if there is a refinement $\mathcal{U}$ of $\mathcal{V}_1$ and $\mathcal{V}_2$ (for each $U_i\in\mathcal{U}$ there are $V_{i_1}\in\mathcal{V}_1$ and $V_{i_2}\in\mathcal{V}_2$ containing $U_i$) and if there are numbers $\delta_i, \delta_j\in \mathbb{T}$ such that

$$\zeta_{i_2 j_2}= \delta_{i} \zeta_{i_1j_1} \delta_j^{-1}$$

whenever $U_{i}\cap U_{j} \neq \emptyset$.

The first cohomology group $H^1(\Omega_K; \mathbb{T})$ consists of these equivalence classes. Every element in $H^1(\Omega_K; \mathbb{T})$ determines a line bundle. The character group $\Pi(\Omega_K)^\star$ and $H^1(\Omega_K; \mathbb{T})$ are canonically isomorphic, see Theorem 1B, Ch2 \cite{Hasumi}. 

An analytic section $f$ of a line bundle $\zeta$ with representative $(\{\zeta_{ij}\}, \{V_i\})_{{i,j\in I}}$ is an equivalence class with a representative $(\{f_i\}, \{V_i\})_{{i\in I}}$ where each $f_i$ is an analytic function on $V_i$ such that
$$f_i(z)= \zeta_{ij} f_j(z)$$
in $V_i\cap V_j$. Let $f_{i_\mathcal{O}}$ be a branch of $f$ at $\mathcal{O}$. Then the analytic continuation of $f_{i_\mathcal{O}}$ along each closed curve issuing from $\mathcal{O}$ is possible, see Theorem 2B, Ch 2, \cite{Hasumi}. One can construct the corresponding multiplicative analytic function using this. Conversely, given a multiplicative function, we can determine the corresponding analytic section, see Theorem 2C, Ch. 2, \cite{Hasumi}.

We denote by $\mathcal{H}_q(\Omega_K, \zeta)$  the multiplicative analytic functions $F$ whose character is $\zeta$ for which $|F|^q$ has a harmonic majorant and $\mathcal{H}_\infty(\Omega_K, \zeta)$ means $|F|$ is bounded. It is not difficult to see that, $1\leq p\leq q \leq \infty$ implies that $\mathcal{H}_q(\Omega_K)\subset \mathcal{H}_p(\Omega_K)$.

As in the proof of Theorem 1.4 in \cite{Chszy} we need the function $B_{\Omega_K}$ to prove Theorem \ref{ratfreq}. We can find a local harmonic conjugate to $-g_{\Omega_K}(z)$ for each $z$. Therefore, the equation $|B_{\Omega_K}(z)|= e^{-g_{\Omega_K}(z)}$ determines a multiplicative analytic function on $\Omega_K$ up to a multiplicative constant. We fix it by requiring 

\begin{equation}\label{green}
B_{\Omega_K}(z)= \mathrm{Cap}(K)/z+\mathrm{O}(|z|^{-2})
\end{equation}

near $\infty$.

Let $c$ be a rectifiable curve on $\Omega_K$ such that $c$ winds once around $L\subset K$ and around no other points of $K$, then the 
change of phase of $B_e$ around $c$ is given by $e^{-2\pi i \mu_K(L)}$, see Theorem 2.7 in \cite{Csz1}. Using this we can determine $\zeta_{B_{\Omega_K}}(\cdot)$. Let us denote the character of $B_{\Omega_K}^n$ by $\chi_K^n$ for simplicity. 

Multiplication of two characters $\zeta_1$ and $\zeta_2$ in $\Pi (\Omega_K)^\star$ is defined as pointwise multiplication: $(\zeta_1\zeta_2)(c)=\zeta_1(c) \zeta_2(c)$. This makes $\Pi (\Omega_K)^\star$ an abelian group. Let us equip $\Pi (\Omega_K)$ with discrete topology. Then  $\Pi (\Omega_K)^\star$ is a compact metrizable space with the topology of pointwise convergence since $\Pi (\Omega_K)$ is countable. The map $T\zeta  := \chi_K \zeta$ is ergodic with respect to the Haar measure if and only if $\{\chi_K^n\}_{n=-\infty}^\infty$ is dense in $\Pi (\Omega_K)^\star$, see 
Theorem 1.9 in \cite{walters}. If $\{\chi_K^n\}_{n=-\infty}^\infty$ is dense then $\{\chi_K^n\}_{n=0}^\infty$ is also dense, see p. 132 in \cite{walters}. This fact is used in the proof of Theorem \ref{ratfreq}.

When $K\subset \mathbb{R}$,  $\{\chi_K^n\}_{n=-\infty}^\infty$ is dense in $\Pi(\Omega_K)^*$ if and only if the following condition is satisfied, see Section 1 in \cite{Chszy}:
Suppose that for each  decomposition $K= K_1 \cup \ldots \cup K_l$  into closed disjoint sets and rational numbers $\{q_j\}_{j=1}^{l-1}$ we have 

$$\sum_{j=1}^{l-1} q_j \mu_K(K_j)\neq 0.$$

Let us review the results of Widom in \cite{widom2} on asymptotics of orthogonal polynomials on Jordan curves. We say that $\gamma: [0,1]\rightarrow \mathbb{C}$ is a Jordan curve if it simple and closed. A rectifiable Jordan curve $\gamma$ is $C^{2+}$ if $\gamma$ is $C^2$, $\gamma^\prime (t)\neq 0$ and the second derivative of $\gamma$ satisfies a Lipschitz condition with some positive exponent.

Let $\Gamma = \{\Gamma_1, \ldots, \Gamma_p\}$ denote the union of a finitely many $C^{2+}$ Jordan curves which are mutually exterior to each other. We denote the critical points of $g_{\Omega_\Gamma}$ by $z_1^\star,\ldots ,z_{p-1}^\star$ counting multiplicity. Let $ds$ denote the arc measure on $\Gamma$. Then 
\begin{equation*}
d\mu_\Gamma(s)= \frac{1}{2\pi} \frac{\partial g_{\Omega_\Gamma}(s)}{\partial \textbf{n}} ds,
\end{equation*}

where $\textbf{n}$ denotes the inner normal derivative relative to $\Omega_K$, see p. 121 in \cite{saff}.

Let $\mu(s)= f(s) ds$ be a unit Borel measure on $\Gamma$ such that 

\begin{equation}\label{diri2}
\int_\Gamma \log{f}\, d\mu_\Gamma > -\infty.
\end{equation}

Let $R(z)$ denote the multiplicative analytic function on $\Omega_\Gamma$, without any zeros or poles and having non-tangential boundary values $ |R(z)|= f(z)$. By \eqref{diri2}, $H_{\Omega_\Gamma}(\cdot;\log{f})$  is a harmonic function on $\Omega_\Gamma$. It has (non-tangential) boundary values $\log{f}$ since the boundary is $C^{1+}$ (see Theorem 3B, Ch4 \cite{Hasumi} and Theorem 5.1 in \cite{taylor}) . We can find a local harmonic conjugate for $H_{\Omega_\Gamma}(\cdot;\log{f})$ at each $z\in \Omega_\Gamma$. Thus, $R$ can be written as

$$R(z)= \exp{[H_{\Omega_\Gamma}(z;\log{f})+ i {\tilde{H}_{\Omega_\Gamma}}(z;\log{f})]}.$$

Here, $R$ is determined up to a constant of absolute value $1$. We can uniquely determine $R$ by requiring $R(\infty)>0$. Thus (p. 155 in \cite{widom2}), $\log{R{(\infty)}}= H_{\Omega_\Gamma}(\infty;\log{f})$ and 

$$R(\infty)= \exp\left(\int_\Gamma \log{f}\, d\mu_\Gamma\right).$$

At each $z\neq w$ in $\Omega_\Gamma$ we can find a harmonic conjugate $\tilde{g}_{\Omega_\Gamma}(z;w)$ for $g_{\Omega_\Gamma}(z;w)$. Let $\Phi(z,w) := \exp[g_{\Omega_\Gamma}(z;w)+i \tilde{g}_{\Omega_\Gamma}(z;w)]$. We define $\Phi(z):=\Phi(z;\infty)$.


Let us define $f_0$ as follows:
\begin{equation}\label{supereq1}
 d\mu(s)= f(s) ds= f_0(s)\frac{\partial g_{\Omega_\Gamma}(s)}{\partial \textbf{n}} ds= 2\pi f_0(s) d\mu_{\Gamma}(s).
\end{equation}
That is,

\begin{equation}\label{supereq2}
f_0(s)= \frac{1}{2\pi}\frac{d\mu}{d\mu_{\Gamma}}(s).
\end{equation}

By \eqref{supereq1}, $R$ can be written as a product of two multiplicative  analytic functions whose boundary values (of absolute value) are $f_0$ and $\frac{\partial g_{\Omega_\Gamma}}{\partial \textbf{n}} $ respectively.

We denote the multiplicative analytic function without any zeros or poles on $\Omega_\Gamma$ which on $\Gamma$  has absolute value $f_0$ by $R_0$. Since $\frac{\partial g_{\Omega_\Gamma}}{\partial \textbf{n}}$ is bounded (see e.g. Corollary 4.7, Ch. 4 in \cite{garnett}) by positive numbers from below and above on $\Gamma$, by the assumption \eqref{diri2} and \eqref{supereq1},  it is given by the formula
\begin{equation}
R_0(z)= \exp{[H_{\Omega_\Gamma}(z;\log{f_0})+ i {\tilde{H}_{\Omega_\Gamma}}(z;\log{f_0})]}.
\end{equation}

We require that $R_0(\infty)>0$. Thus,
\begin{equation}\label{import}
R_0(\infty)= \exp\left(\int_\Gamma \log{f_0}\, d\mu_\Gamma\right).
\end{equation}

The multiplicative analytic function without zeros or poles in $\Omega_\Gamma$ which on $\Gamma$  has absolute value $\frac{\partial g_{\Omega_\Gamma}}{\partial \textbf{n}} $ is 
$\Phi^\prime(z)\prod_1^{p-1} \Phi(z; z_j^\star)$, see p. 175 in \cite{widom2}. Hence 

$$R(z)= R_0(z) \Phi^\prime(z)\prod_1^{p-1} \Phi(z; z_j^\star).$$

We define 
$$ \nu(f,\zeta):= \displaystyle \inf_F \int_\Gamma |F(s)|^2 f(s) ds$$

where infimum is taken among the functions in $\mathcal{H}_2(\Omega_\Gamma, \zeta)$ such that $|F(z)^2 R(z)|$ has a harmonic majorant in $\Omega_\Gamma$.
Then (see p. 176 in \cite{widom2})

\begin{equation*}
\nu(f,\zeta)= 2\pi R_0(\infty)\exp\{\sum_{j=1}^{p-1}g_{\Omega_\Gamma}(z_j^\star)-\sum_{j=1}^{p-1}\epsilon_j g_{\Omega_\Gamma}(z_j)\}
\end{equation*}

where $\epsilon=\mp 1$ and $z_j$ are uniquely determined by a set of equations, see p. 169 in \cite{widom2}. But the infimum (keeping $f$ fixed) of $\nu(f,\zeta)$ is assumed when $z_j=z_j^\star$ and $\epsilon_j=1$ for $j=1,\ldots,p-1$, see p. 171.

Therefore

\begin{equation}\label{import2}
\inf_\zeta \nu(f,\zeta)= 2\pi R_0(\infty).
\end{equation}
If we denote $\frac{d\mu}{d\mu_{\Gamma}}$ by $h$ then by \eqref{supereq2}, \eqref{import}, \eqref{import2} we get

\begin{equation*}
\inf_\zeta \nu(f,\zeta)= \exp{\left\{ \int \log{h}\, d\mu_{\Omega_\Gamma} \right\} }.
\end{equation*}

We have the following lower bound (see p. 216 in \cite{widom2}) for the $W_n$'s:

\begin{equation*}
(W_n(\mu))^2\geq \nu(f,\chi_\Gamma^n).
\end{equation*}

Hence for all $n\in\mathbb{N}$, we have

\begin{equation}\label{harika}
(W_n(\mu))^2\geq \nu(f,\chi_\Gamma^n)\geq \inf_\zeta \nu(f,\zeta)= \exp{\left\{ \int \log{h}\, d\mu_{\Omega_\Gamma} \right\} }.
\end{equation}

Hence the lower bound in Theorem \ref{maintheo} is compatible with the lower bound obtained in Widom's paper \cite{widom2} on a system of $C^{2+}$ Jordan curves.

Let $K$ be a non-polar compact subset of $\mathbb{C}$. Then we call $\{K_n\}_{n=1}^\infty$ a $C^{2+}$ exhaustion of $\Omega_K$ if $(K_n)_{n+1}^\infty$ is increasing sequence of domains such that

\begin{enumerate}[(a)]
	\item $\partial K_n$ consists of finitely many non-intersecting $C^{2+}$ Jordan curves.
	\item $\overbar{K_n} \subset {K_{n+1}}$.
	\item $\cup K_n= \Omega_K$.
\end{enumerate}

We can find always find a $C^{2+}$ exhaustion for $\Omega_K$, see VII. 4.4 in \cite{conway} or Ch.2, 12D  in \cite{ahlfors}. The following result concerning the harmonic measures is used next section, see Theorem 10.9 in Section 21.11 in \cite{Conway2} for the proof:

\begin{theorem}\label{weak}
	Let $K$ be a non-polar compact subset of $\mathbb{C}$ and $(K_n)_{n=1}^\infty$ be a $C^{2+}$ exhaustion of $\Omega_K$. Then $w_{K_n}(z;\cdot)\rightarrow w_{\Omega_K}(z;\cdot)$ in the weak-star sense.
\end{theorem}

\section{Proofs}
\begin{proof}[Proof of Theorem \ref{maintheo}]
	Let $P_n(z):= \prod_{j=1}^n (z-\tau_j)$. Then
	\begin{align}
	\left(\int |P_n|^2 h\, d\mu_K+\int |P_n|^2 d\mu_s\right)^{1/2} &\geq \left(\int |P_n|^2 h\, d\mu_K \right)^{1/2}\label{babba11}\\
	&= e^{\log \left(\int |P_n|^2 h\, d\mu_K\right)^{1/2}}\label{babba}\\		
	&\geq e^{\int \log{(|P_n| h^{1/2})}\,d \mu_K } \label{baba1} \\
	&= e^{\frac{1}{2} \int \log{h}\, d\mu_K }\,\, e^{\int \sum_{j=1 }^n \log |z-\tau_j| d\mu_K(z)} \label{baba2}\\
	&\geq e^{\frac{1}{2} \int \log{h}\, d\mu_K }\,\, \mathrm{Cap(K)}^n \label{baba3}.
	\end{align}
	
	Here, \eqref{baba1} follows from Jensen's inequality and \eqref{baba3} holds since
	$$\int \log|z-\tau|\,d\mu_K(z) \geq \log{\mathrm{Cap}(K)}$$ for all $\tau\in\mathbb{C}$ by Frostman's theorem, see Theorem 3.3.4 (a) in \cite{Ransford}. We obtain the desired inequality by squaring the left hand side of \eqref{babba11} and \eqref{baba3} and using the assumption that $\mathrm{Cap(supp(\mu))}= \mathrm{Cap(supp(\mu_K))}$.

\end{proof}

\begin{proof}[Proof of Corollary \ref{equi}]
	We obtain $\inf_{n\in\mathbb{N}} W_n(\mu_K)\geq 1$ by letting $h\equiv 1$ and $\mu_s=0$ in Theorem \ref{maintheo}.
	
	The proof of the second part of the corollary is quite straightforward. Since $|z|=1$ on the unit circle, we get $\int |z|^{2n} d\mu_{\mathbb{T}}(z)=1$ for all $n$. In addition, $P_n(z;\mathbb{T})= z^n$ and $\mathrm{Cap}({\mathbb{T}})=1$. Thus $W_n(\mu_\mathbb{T})=1$ for all $n\in\mathbb{N}$.
\end{proof}

The following characterization of the Parreau-Widom condition is due to Widom (see Theorem 1 in \cite{Widomhp} and also Section 2B, in Ch. 5 in \cite{Hasumi}):

\begin{theorem}\label{widddd}
	Let $K$ be a regular compact subset of $\mathbb{C}$. Then $\Omega_K$ is a Parreau-Widom domain if and only if $\mathcal{H}_2(\Omega_K, \zeta)\neq \{0\}$ for all $\zeta\in \Pi(\Omega_K)^*$.
\end{theorem}

For a multivalued function $F\in \mathcal{H}_p(\Omega_K)$ we denote the least harmonic majorant for $|F|^p$ by $\mathrm{LHM}(|F|^p)(\cdot).$ The function $|F|^p$ is subharmonic. Let $K_n$ be a $C^{2+}$ exhaustion of $\Omega_K$. Then as a consequence of Harnack's theorem (see Theorem 1.3.9 in \cite{Ransford}), we get (see e.g. eq. (2.1.2) in \cite{rudin})
\begin{equation*}
\mathrm{LHM}(|F|^p)(z)=\lim_{n\rightarrow\infty} \int |F|^p dw_{K_n}(z;\cdot).
\end{equation*}

In addition if $|F|^p$ can be extended continuously to $\overbar{\Omega_K}$ then $\mathrm{LHM}(|F|^p)(z)= H_{\overline{\Omega_K}}(z;|F|^p)$. since $dw_{ K_n}(z;\cdot).\rightarrow dw_{\Omega_K}(z;\cdot).$

In the proof of Theorem \ref{ratfreq} we use ideas from Theorem 1.4 in \cite{Chszy} and from the proof of Theorem in 5A, Ch. 5 (the main arguments of the proof can also be found in Theorem 3 in \cite{Widomhp}).) in \cite{Hasumi}.

\begin{proof}[Proof of Theorem \ref{ratfreq}]
Let $M:=\sup_n {W_n(\mu_K)}$ and $\chi\in \Pi(\Omega_K)^*$. Then there is a subsequence $n_j\rightarrow\infty$ such that $\chi_K^ {n_j}\rightarrow \chi$.
	Let $F_j(z):= \frac{P_{n_j}(z;\mu_K) B_{\Omega_K}^{n_j}(z)}{\mathrm{Cap}(K)^{n_j}}$. For each $j$, $F_j\in \mathcal{H}_{\infty}(\Omega_K)$ because by the maximum principle
	$$ \|F_j\|_{\Omega_K}\leq \sup_{z\rightarrow \partial \Omega_K}|F_j(z)|= \sup_{z\in \partial \Omega_K}{|P_{n_j}(z;\mu_K)|} \mathrm{Cap}(K)^{-n_j}.$$
	Hence we also have $F_j\in \mathcal{H}_2(\Omega_K)$. Note that $|F_j|^2$ can be extended continuously to $\overbar{\Omega_K}$ since $K$ is regular. Thus,
	\begin{equation}\label{bound}
	\mathrm{LHM}(|F_j|^2)(\infty)= \frac{\int |P_{n_j}(z;\mu_K)|^2 d\mu_K}{\mathrm{Cap}(K)^{2 n_j}}= (W_{n_j}(\mu_K))^2\leq M^2.
	\end{equation}

	Let $\{V_\alpha^\prime\}$ be an open covering of $\Omega_K$ such that each $V_\alpha^\prime$ is  simply connected,  \\$V_\alpha^\prime=V_\alpha^\prime\cap \Omega_K$ and if $V_{\alpha_1}^\prime\cap V_{\alpha_2}^\prime \neq \emptyset$ then $V_{\alpha_1}^\prime\cup V_{\alpha_2}^\prime$ is included in a simply connected subset of $\Omega_K$. Let  $\{V_\alpha\}$ be a refinement of $\{V_\alpha^\prime\}$ such that it is countable, each $\overline{V_\alpha}$ is compact and $\overline{V_\alpha}\subset {V_\alpha^\prime}$ for each $\alpha$. We can find (by the same argument used in p. 133 in \cite{Hasumi} for $f_\alpha$) a representative $(\{F_{j\alpha}\}, \{V_\alpha^\prime\})$ of the analytic section corresponding to $F_j$. Since $\mathrm{LHM}(|F_j|^2)(\cdot)$ is a positive harmonic function on $\Omega_K$ and the inequality \eqref{bound} holds, in view of Harnack's inequality, $(\mathrm{LHM}(|F_j|^2)(\cdot))_{j=1}^\infty$ is uniformly bounded  above by a positive number on each $V_\alpha$. By Montel's theorem, we can find a subsequence of $(F_{{j\alpha}})_{j=1}^\infty$ which is uniformly convergent on $V_\alpha$. Since $\{V_\alpha\}$ is a countable covering, by a diagonalization argument, we can find a subsequence $\{F_{j(k)}: k=1, 2 \ldots\}$ such that, for each $\alpha$, $\{F_{j(k),\alpha}: k=1,2,\ldots \}$ converges uniformly on ${V_\alpha}$. We denote the limit of this subsequence on $V_\alpha$ by $f_\alpha$. It is clear that $(\{f_{\alpha}\}, \{V_\alpha\})$ represents an analytic section. Let  $f$ be the corresponding multivalued analytic function. Then $\zeta_{f}(c)= \lim_{k\rightarrow} \zeta_{F_{j(k)}}(c)= \chi(c)$. 
	
    Note that $|F_{j(k)}(\infty)|=1$ for all $k$ by \eqref{green}. Therefore
    
    \begin{equation}\label{alp}
    |f(\infty)|=1.
    \end{equation}

	 It remains to show that $|f|^2$ has a harmonic majorant. Note that $|F_{j(k)}(z)| \rightarrow |f(z)|$ uniformly on each compact subset of $\Omega_K$. Fix a positive integer $n$. Let $\epsilon>0$. Then there is a number $k_0$ with $n< j(k_0)$ such that 
	 
	 \begin{equation}\label{epsli}
	 ||F_{j(k)}(z)|^2- |f(z)|^2|<\epsilon
	 \end{equation}
	 is satisfied on $\partial K_n$ for all $k\geq k_0$.
	 
	 Let us denote the least harmonic majorant of a function $G$ restricted to a region $E$ by $\mathrm{LHM}_E (G)(\cdot)$. Then by \eqref{epsli}, for each $k\geq k_0$,
	 
	 \begin{equation*}
	 \mathrm{LHM}_{K_n} (|f|^2) (z) \leq \mathrm{LHM}_{K_n} (|F_{j(k)}|^2 +\epsilon)(z) \leq \mathrm{LHM}_{K_n} (|F_{j(k)}|^2) (z) + \epsilon
	 \end{equation*}
	 for $z\in K_n$.
	 
	 Clearly,  $\mathrm{LHM}_{K_n} (|F_{j(k)}|^2) (z) \leq \mathrm{LHM} (|F_{j(k)}|^2) (z)$. Since $\epsilon$ is arbitrary, we get 
	 
	 \begin{equation*}
	  \mathrm{LHM}_{K_n} (|f|^2) (z) \leq \limsup_{k\rightarrow\infty } \mathrm{LHM} (|F_{j(k)}|^2) (z).
	 \end{equation*}
	 
	 By \eqref{bound} and Harnack's inequality, there is a constant $C(z)$ depending only on $z$ such that $\limsup_{k\rightarrow\infty } \mathrm{LHM} (|F_{j(k)}|)^2 (z)\leq C(z) M^2$. Hence 
	 $\mathrm{LHM}_{K_n} (|f|^2) (z) \leq C(z)M^2$. Since $n$ is arbitrary, $\mathrm{LHM}_{K_r} (|f|^2) (z) \leq C(z)M^2$ for all $r\in\mathbb{N}$.
	 
	  For any fixed $z$, let $l$ be an integer such that $z\in K_l$. Then $(\mathrm{LHM}_{K_n} (|f|^2) (z))_{n=l}^\infty$ is an increasing sequence bounded by $C(z) M^2$. Let $H(z):=\lim _{n\rightarrow\infty} (\mathrm{LHM}_{K_n} (|f|^2) (z))$. Then by Harnack's theorem (see Theorem 1.3.9 in \cite{Ransford}) $H$ is a harmonic function on $\Omega_K$. Clearly $ \mathrm{LHM} (|f|^2) (z) \leq H(z) \leq C(z) M^2$. Thus $f$ is in $\mathcal{H}_2(\Omega_K,\chi)$. It is also non-zero by \eqref{alp}. Since $\chi$ is arbitrary, this proves that $\Omega_K$ is a Parreau-Widom domain by Theorem \ref{widddd}.
	 
		\end{proof}

\begin{proof}[Proof of Corollary \ref{ratfreq2}]
	
	Suppose that $(W_n(\mu_K))_{n=1}^\infty$ is bounded. Then by Theorem \ref{ratfreq}, $\Omega_K$ is Parreau-Widom and by Theorem 1.4 in \cite{Csz1}, $(M_{n,K})_{n=1}^\infty$ is bounded. This proves the first implication.
	
	Suppose that $(M_{n,K})_{n=1}^\infty$ is bounded. Note that 
	
	\begin{equation}\label{chebbb}
		\|P_n(\cdot;\mu_K)\|_{L^2(\mathbb{C}; \mu_K )}\leq \|T_{n,K}\|_{L^2(\mathbb{C}; \mu_K) } \leq \|T_{n,K}\|_K.
	\end{equation}

	The inequality on the left follows from \eqref{infli} and for the second inequality, we refer the reader to the proof of Corollary 1.2 in \cite{simon1}. Thus,
	
	$W_n(\mu_K)\leq M_n(K)$ and this implies that  $(W_n(\mu_K))_{n=1}^\infty$ is also bounded.

\end{proof}

\section*{Acknowledgement}
I would like to thank Barry Simon for the helpful suggestions regarding Theorem \ref{maintheo}. 

I am grateful to Vilmos Totik for sharing his observations on the proof of Theorem \ref{maintheo}. It was his his idea to use Frostman's theorem in the proof.

\section*{Appendix}
We give another proof of Theorem \ref{maintheo} with $\mu_s=0$ using a completely different method. 

To prove the theorem, we use the fact that this result holds on a system of Jordan curves, see \eqref{harika}. We combine some of the results obtained in \cite{widom2}. In \cite{widom2}, the Radon-Nikodym derivative with respect to the equilibrium measure is considered to explore the properties of the Szeg\H{o} kernel $K(z,z_0)$, see p. 175. What is new, here, is  the observation that once the asymptotics of orthogonal polynomials are expressed in terms of the Radon-Nikodym derivative with respect to the equilibrium measure (instead of the arc measure) then the lower bound for the $W_n$'s is given in terms of the Szeg\H{o} integral and it is independent of the Parreau-Widom sum. 

Replacing the arc measure by the equilibrium measure of a system of Jordan curves enables us to approximate a measure supported on an arbitrary regular compact subset of $\mathbb{C}$ by measures supported on Jordan curves. Besides, the same lower bound holds for the Szeg\H{o} integrals of these measures, see Lemma \ref{temellem}.
\subsection*{Proof of Theorem \ref{maintheoo}: Continuous Case}

\begin{lemma} \label{temellem}
	Let $K$ be a regular compact subset of $\mathbb{C}$ and let $f$ be a non-negative continuous real-valued function on $\partial\Omega_K$ and let $\mu$ be a unit Borel measure on $\partial\Omega_K$ such that $\mu = f\, d\mu_K$. Suppose that $M:=\int \log{f}\, d\mu_K >-\infty.$  Let $(K_n)_{n=1}^\infty$ be a $C^{2+}$ exhaustion of $\Omega_K$. Then the sequence of measures $\mu_n:= H_{\Omega_K}(z;f)_{|\partial K_n}d\mu_{\partial K_n}$ satisfies the following properties:
	\begin{enumerate}[(a)]
		\item $\mu_n\rightarrow \mu$ (weak star)
		\item $\int \log{H_{\Omega_K}(z;f)_{|\partial K_n}}\, d\mu_{\partial K_n}\geq M$.

	\end{enumerate}
\end{lemma}

\begin{proof}
	\begin{enumerate}[(a)]
		\item Since $f$ is continuous and $K$ is regular, $H_{\overbar{\Omega_K}}(\cdot;f)$ is continuous on $\overbar{\Omega_K}$. Hence $\mu_n\rightarrow \mu$ follows from Theorem \ref{weak}.
		
		\item Fix $n\in\mathbb{N}$.	Let $f_n(z):={H_{\Omega_K}(z;f)_{|\partial K_n}}$. Note that ${H_{\Omega_K}(\infty;f)}=\int f d\mu_K=1$ since $\mu$ is a unit measure. Since  $H_{\Omega_K}(\cdot;f)$ is harmonic on $K_n$ and it is continuous on $K_n\cup \partial K_n$ we have $H_{K_n}(z;f_n)= {H_{\Omega_K}(z;f)}$ for $z\in K_n$ by  uniqueness of the solution of the Dirichlet problem. This implies that $\int f_n d\mu_{\partial K_n} =H_{K_n}(\infty;f_n)= H_{\Omega_K}(\infty;f)=1$. Hence $\mu_n$ is a unit Borel measure for each $n\in\mathbb{N}$.
		
		Since $M>-\infty$, we have $f>0$ $\mu_K$ a.e. and thus $f>0$ $w_{\Omega_K}(z;\cdot)$ a.e. for any $z\in \Omega_K$, see e.g. Corollary 4.3.5 in \cite{Ransford}. Let $G(z):=  H_{\Omega_K}(z;\log{f})$ on $\Omega_K$. Using the uniqueness of the solution of the Dirichlet problem as in the above paragraph, we see that $\int G_{|\partial K_n} d\mu_{\partial K_n}= \int \log{f}\, d\mu_K=M$. In view of Jensen's inequality, we get
		\begin{equation}\label{aaaa}
		G(z)= \int_{\partial \Omega_K} \log{f}  dw_{\Omega_K}(z;\cdot)\leq \log \int_{\partial \Omega_K}  f dw_{\Omega_K}(z;\cdot)=\log H_{\Omega_K}(z;f)
		\end{equation}
		
		for $z\in \Omega_K$. Integrating the restrictions of the first and the last terms to $\partial K_n$ in \eqref{aaaa} with respect to $\mu_{\partial K_n}$,
		
		\begin{equation*}
		M=\int G_{|\partial K_n} d\mu_{\partial K_n}\leq \int \log f_n d\mu_{\partial K_n},
		\end{equation*}
		we prove part $(b)$.
	\end{enumerate}

\end{proof}
Let us prove Theorem \ref{maintheoo} for continuous $h$:

\begin{theorem} \label{maintheo2}
	Let $K$ be a regular compact subset of $\mathbb{C}$ and let $\mu$ be a unit Borel measure on $\partial\Omega_K$ such that 
	\begin{itemize}
		\item $d\mu= h \,d{\mu}_K$ where $h$ is a continuous real valued function on $\partial \Omega_K$.
		\item $M:=\int \log{h}\, d\mu_K>-\infty.$

	\end{itemize}
	
	Then $\inf_{s\in\mathbb{N}} (W_s(\mu))^2 \geq e^M.$
\end{theorem}

\begin{proof}
	Let $(K_n)_{n=1}^\infty$ be a $C^{2+}$ exhaustion of $\Omega_K$. Let $\mu_n$ be as in the statement of Lemma  \ref{temellem} and $s\in\mathbb{N}$. Then
	\begin{align}
	(W_s(\mu))^2 &= \frac{\int |P_s(\cdot;\mu)|^2 d\mu}{\mathrm{Cap}(K)^{2s}}\\		
	&=\lim_{n\rightarrow\infty} \frac{\int |P_s(\cdot;\mu)|^2 d\mu_n}{\mathrm{Cap}(K)^{2s}} \label{a1} \\
	&\geq \limsup_{n\rightarrow\infty} \frac{\int |P_s(\cdot;\mu)|^2 d\mu_n}{\mathrm{Cap}(\partial K_n)^{2s}}\label{a2}\\
	&\geq \limsup_{n\rightarrow\infty} (W_s(\mu_n))^2\label{a3}\\
	&\geq e^{M}.\label{a4}
	\end{align}
	where \eqref{a1} follows from Lemma \ref{temellem} (a), \eqref{a2} follows from monotonicity of capacity, \eqref{infli} implies \eqref{a3} and \eqref{harika} implies \eqref{a4}. This completes the proof.
\end{proof}

\subsection*{Proof of Theorem \ref{maintheoo}: General Case}

\begin{theorem} \label{maintheo3}
	Let $K$ be a regular compact subset of $\mathbb{C}$ and let $\mu$ be a unit Borel measure on $\partial\Omega_K$ such that 
	\begin{itemize}
		\item $d\mu= h \,d{\mu}_K$ where $h$ is a function bounded below (almost everywhere with respect to $\mu_K$) by a number $C>0$ on $\partial\Omega_K$.
		\item $M:=\int \log{h}\, d\mu_K.$

	\end{itemize}
	
	Then $\inf_{s\in\mathbb{N}} (W_s(\mu))^2 \geq e^M.$
\end{theorem}

\begin{proof}
	There is a sequence $(q_n)_{n=1}^\infty$ of lower semicontinuous   functions on $\partial \Omega_K$ such that $q_n\geq h$ and  $\lim_{n\rightarrow\infty} \int |q_n-h| d\mu_K= 0$, see eq. (3.8) in Appendix A.3, \cite{saff}. For each $n$ there is an increasing sequence of real valued continuous functions $(f_{n,k})_{k=1}^\infty$ on $\partial \Omega_K$ such that $f_{n,k}(z)\rightarrow q_n(z)$ pointwise as $k\rightarrow\infty$, see e.g. p.1, \cite{saff}. By the monotone convergence theorem 
	$$\lim_{k\rightarrow\infty}\int |f_{n,k} - q_n | \,d\mu_K=0$$
	for each $n$. Therefore, we can find a sequence of real valued continuous functions $(f_{n, k(n)})_{n=1}^\infty$ on $\partial \Omega_K$ such that 
	$$ \lim_{n\rightarrow\infty} \int |h- f_{n, k(n)}|\, d\mu_K=0.$$
	
	Let  $t_n(z):= \max\{f_{n, k(n)}(z), C\}$ at each $z\in \partial\Omega_K$. Then $t_n$ is continuous on $ \partial\Omega_K$ and 
	
	\begin{equation}\label{muthis}
	\int |t_n-h| d\mu_K\rightarrow 0
	\end{equation}
	
	as $n\rightarrow \infty$. Besides, $$L_n:=\int t_n\, d\mu_K\rightarrow 1$$ as $n\rightarrow\infty$. 
	Let $h_n:= t_n/L_n$. Then $\eta_n:=h_n d\mu_K$ is a unit Borel measure for each $n\in\mathbb{N}$ and \\ $\int |t_n-h_n| d\mu_K\rightarrow 0$. By \eqref{muthis}, this implies that $\lim_{n\rightarrow\infty} \int |h_n-h| d\mu_K= 0$. Thus $h_n d\mu_K\rightarrow hd\mu_K$ in the weak star sense.  There is a  $D>0$ such that both $h$ and each $h_n$ are bounded below by $D$ on $ \partial\Omega_K$. It follows that
	\begin{equation}\label{muthis2}
	\lim_{n\rightarrow\infty} \int \log{h_n}\, d\mu_K=  \int \log{h}\, d\mu_K
	\end{equation}
	since $\log$ is Lipschitz continuous on $[D,\infty)$. 
	
	Let $s\in\mathbb{N}$.
	Then
	
	\begin{align}
	(W_s(\mu))^2 &= \frac{\int |P_s(\cdot;\mu)|^2 d\mu}{\mathrm{Cap}(K)^{2s}}\\		
	&=\lim_{n\rightarrow\infty} \frac{\int |P_s(\cdot;\mu)|^2 d\eta_n}{\mathrm{Cap}(K)^{2s}} \label{b1} \\
	&\geq \limsup_{n\rightarrow\infty} (W_s(\eta_n))^2\label{b2}\\
	&\geq e^{M}.\label{b3}
	\end{align}
	Here, \eqref{b1} follows from the fact that $\eta_n\rightarrow \mu$,   \eqref{b2} follows from  \eqref{infli} and \eqref{b3} follows from \eqref{muthis2} and Theorem \ref{maintheo2}. This completes the proof.
\end{proof}

\begin{theorem} \label{maintheoo}
	Let $K$ be a regular compact subset of $\mathbb{C}$ and let $\mu$ be a unit Borel measure supported on $\partial\Omega_K$. Let $h$ be a non-negative measurable function on $\partial \Omega_K$ such that 
	\begin{itemize}
		\item $d\mu= h \,d{\mu}_K$.
		\item $M:=\int \log{h}\, d\mu_K >-\infty.$
		
	\end{itemize}
	
	Then $\inf_{s\in\mathbb{N}} (W_s(\mu))^2 \geq e^M.$
\end{theorem}

\begin{proof}
	Let $s\in\mathbb{N}$.
	Let $t_n(z):= h(z)+ 1/n$. Then
	\begin{equation}\label{mli}
	\int \log{t_n} d\mu_K\geq M.
	\end{equation}
	
	and
	
	\begin{equation}\label{monot}
	\lim_{n\rightarrow\infty}  \int {t_n} |P_s(\cdot;\mu)|^2 d\mu_K =  \int  |P_s(\cdot;\mu)|^2 h\, d\mu_K.
	\end{equation}
	
	Let $\mu_n:= \frac{t_n}{1+1/n} d\mu_K$.Then $\frac{t_n}{1+1/n} $ is bounded below by $\frac{1/n} {1+1/n}$ and $\mu_n$ is a unit Borel measure for each $n$. By \eqref{monot},
	
	\begin{equation}\label{monot2}
	\lim_{n\rightarrow\infty}  \int |P_s(\cdot;\mu)|^2\, d\mu_n=  \int  |P_s(\cdot;\mu)|^2 h\, d\mu_K.
	\end{equation}
	
	In addition, we have 
	
	\begin{equation*}
	\limsup_{n\rightarrow\infty} \int \log{\frac{t_n}{1+1/n}}\, d\mu_K\geq M.
	\end{equation*}
	by \eqref{mli}. Thus, in view of Theorem \ref{maintheo3}, we get
	
	\begin{equation}\label{mli3}
	\limsup_{n\rightarrow\infty} (W_s (\mu_n))^2\geq e^M.
	\end{equation}
	
	Using \eqref{monot2} in \eqref{c2}, \eqref{infli} in \eqref{c3} and \eqref{mli3} in \eqref{c4} we deduce that
	
	\begin{align}
	(W_s(\mu))^2 &=\frac{ \int  |P_s(\cdot;\mu)|^2 h\, d\mu_K}{\mathrm{Cap}(K)^{2s}}\\
	&= \lim_{n\rightarrow\infty}\frac{ \int |P_s(\cdot;\mu)|^2\, d\mu_n}{\mathrm{Cap}(K)^{2s}}\label{c2}\\
	&\geq  \limsup_{n\rightarrow\infty} (W_s (\mu_n))^2\label{c3}\\
	&\geq e^M.\label{c4}
	\end{align}
	The proof is complete.
\end{proof}

\end{document}